\documentclass[11pt]{article}
\usepackage{tgpagella}
\usepackage[utf8]{inputenc}
\usepackage[T1]{fontenc}
\usepackage{amssymb}
\usepackage{amsfonts}
\usepackage{amsmath}
\usepackage{amsthm}
\usepackage{mathrsfs}
\usepackage{comment}
\usepackage[makeroom]{cancel}
\usepackage[normalem]{ulem}
\usepackage{color}
\newtheorem{theorem}{Theorem}
\newtheorem{proposition}[theorem]{Proposition}
\newtheorem{lemma}[theorem]{Lemma}
\newtheorem{fact}[theorem]{Fact}

\newtheorem{cor}[theorem]{Corollary}

\theoremstyle{definition}
\newtheorem{definition}[theorem]{Definition}
\newtheorem{example}[theorem]{Example}

\newcommand{\PA}{\textnormal{PA}}
\newcommand{\BSigma}{\textnormal{B}\Sigma}

\newcommand{\Q}{\textnormal{Q}}

\newcommand{\set}[2]{\left\{ #1 \ \mid \ #2 \right\} }
\newcommand{\res}{{\upharpoonright}}
\newcommand{\CT}{\textnormal{CT}}
\newcommand{\df}[1]{\textbf{#1}}
\newcommand{\num}[1]{\underline{#1}}

\newcommand{\ind}{\textnormal{ind}}
\newcommand{\INT}{\textnormal{INT}}
\newcommand{\LocColl}{\textnormal{LocColl}}
\newcommand{\LocInd}{\textnormal{LocInd}}
\newcommand{\Coll}{\textnormal{Coll}}
\newcommand{\ElDiag}{\textnormal{ElDiag}}
\newcommand{\val}[1]{{#1}^{\circ}}
\newcommand{\Val}{\textnormal{Asn}}
\newcommand{\tuple}[1]{\langle #1 \rangle}
\renewcommand{\Pr}{\textnormal{Pr}}
\newcommand{\dpt}{\textnormal{dp}}

\newcommand{\Th}{\textnormal{Th}}
\newcommand{\Lang}{\mathscr{L}}
\newcommand{\dom}{\textnormal{dom}}
\newcommand{\Con}{\textnormal{Con}}
\newcommand{\SigmanColl}{\Sigma_n \textnormal{-} \Coll}
\newcommand{\qcr}[1]{\ulcorner #1 \urcorner}
\newcommand{\StrReg}{\textnormal{SRP}}

\newcommand{\LPA}{\mathscr{L}_{\PA}}
\newcommand{\form}{\textnormal{Form}}
\newcommand{\Term}{\textnormal{Term}}
\newcommand{\Sent}{\textnormal{Sent}}
\newcommand{\Var}{\textnormal{Var}}
\newcommand{\FV}{\textnormal{FV}}
\newcommand{\ClTerm}{\textnormal{ClTerm}}
\newcommand{\TermSeq}{\textnormal{TermSeq}}
\newcommand{\ClTermSeq}{\textnormal{ClTermSeq}}
\newcommand{\pa}{\textnormal{PA}}
\newcommand{\TI}{\textnormal{TI}}
\newcommand{\Tr}{\textnormal{Tr}}
\newcommand{\Comp}{\textnormal{Comp}}

\title{Local collection scheme and end-extensions of models of compositional truth}
\author{Mateusz Łełyk, Bartosz Wcisło}

\begin{document}

	\maketitle
\begin{abstract}
 We introduce a principle of local collection for compositional truth predicates and show that it is conservative over the classically compositional theory of truth in the arithmetical setting. This axiom states that upon restriction to formulae of any syntactic complexity, the resulting predicate satisfies full collection. In particular, arguments using collection for the truth predicate applied to sentences occurring in any given (code of a) proof do not suffice to show that the conclusion of that proof is true, in stark contrast to the case of induction scheme. 
 
 We analyse various further results concerning end-extensions of models of compositional truth and collection scheme for the compositional truth predicate.
\end{abstract}

\section{Introduction}

The area of axiomatic truth theories investigates extensions of foundational theories, such as Peano Arithmetic ($\PA$)  with an additional predicate $T$ which is intended to denote the set of (codes of) true sentences. 

One of the canonical examples of these theories is $\CT^-$ (Compositional Truth). It is a theory of truth over $\PA$ whose axioms state that the predicate $T$ satisfies Tarski's compositional conditions for arithmetical sentences. For instance, a disjunction of two sentences is true if either of the disjuncts is. However, we do not assume that the truth predicate satisfies any induction whatsoever. All purely arithmetical formulae satisfy the induction scheme because $\CT^-$ by definition contains the whole $\PA$.

It is a very simple and classical fact that $\CT^-$ with the full induction, called $\CT$, is not conservative over $\PA$. By induction on the length of proofs, we can show that whatever is provable in $\PA$ is true and thus show the consistency of arithmetic. By a theorem of Kotlarski, Krajewski, and Lachlan, $\CT^-$ itself is conservative over $\PA$. In fact, not much induction is needed to yield non-conservativeness. It has been shown in \cite{wcislyk}, Theorem 13, that already $\CT^-$ with induction for $\Delta_0$-formulae proves new arithmetical sentences. 

Richard Kaye asked whether the conservativity result remains true if $\CT^-$ is enriched with full scheme of collection
 for the sentences containing the truth predicate.\footnote{The question was posed on a session of Midlands Logic Seminar, see \cite{kaye-slides}.} It is known that in presence of $\Delta_0$ induction the full schemes of collection and induction are equivalent. However, without the access to this small amount of induction, collection seems to be a very weak principle. If we add the full collection scheme to $\PA^-$ (the theory of the positive part of a discretely ordered semiring), then this extension is $\Pi_1$-conservative over $\PA^-$.\footnote{To our best knowledge, this result first appeared as Exercise 7.7 in \cite{kaye}.} One could hope for an analogous result for the compositional truth predicate. Unfortunately, the methods used by Kaye cannot be implemented directly in the setting of the truth predicate where the conservativity of collection appears to be a much harder problem. In particular, as shown by Smith \cite{smith} there are countable models $M\models \CT^-$ with no proper end-extensions and Kaye's argument rests on the fact that every model of $\PA^-$ can be properly end-extended.

In this paper, we provide a partial answer to the question of Kaye. We introduce a principle of \df{local collection}. It states that if we restrict our compositional truth predicate to sentences of any syntactic depth $c$, the resulting truncated predicate satisfies full collection. We show that the principle of local collection for the compositional truth predicate is conservative over $\PA$.

Already this result shows that there is no full analogy between collection and induction in the setting of truth theories. One could introduce a similar scheme of local induction saying that the truth predicate truncated to sentences of any fixed syntactic depth $c$ satisfies full induction. We could readily check that this weaker form of induction is enough to show that there are no proofs of contradiction in $\PA$, since any given proof $d$ involves only formulae of some bounded syntactic depth $c$, so we can check by induction that all formulae in $d$ are true. This shows that local induction is not conservative over $\PA$, in contrast to local collection.

\section{Preliminaries}

In this section, we present some basic definitions and background results.

\subsection{Arithmetic and coding}

This paper deals with extensions of Peano Arithmetic ($\PA$). This is a theory in the language $\LPA = \{0,S, +, \times\}$ consisting of finitely many basic axioms of Robinson's Arithmetic $\Q$ which essentially say how $+$ and $\times$ can be defined inductively in terms of the successor function, and the induction scheme.

Full induction scheme is equivalent to induction for $\Delta_0$--formulae together with full \df{collection scheme}, $\Coll$. The latter consists of all formulae of the following form (where we allow $\phi(x,y)$ to contain more free variables than just $x,y$):
\begin{displaymath}
\forall x< a \exists y \ \ \phi(x,y) \rightarrow \exists b \forall x < a  \exists y<b \ \ \phi(x,y).
\end{displaymath}
Intuitively, collection scheme expresses that any function with a bounded domain has bounded range. This is clearly true in the natural numbers, since bounded segments of $\mathbb{N}$ are finite and hence the image of any such set is also finite. 
 Induction for $\Delta_0$-formulae is crucial for the equivalence between induction and collection as shown by Kaye.\footnote{As we already indicated, this appears as Exercise 7.7 in \cite{kaye}.}
\begin{theorem}[Kaye] \label{th_kaye}
	$\PA^-$ with full collection scheme (but no induction) is conservative for $\Pi_1$-formulae over $\PA^-$.
\end{theorem}

Peano arithmetic, and its much weaker fragments are capable of representing syntactic notions. Below, we list formulae representing syntactic notions which we will use throughout the paper.

\begin{definition} \label{def_syntactic notions}
	\begin{itemize}
		\item $\Term_{\LPA}(x)$ states that $x$ is (a code of) an arithmetical term.
		\item $\TermSeq_{\LPA}(x)$ states that $x$ is (a code of) a sequence of arithmetical terms.
		\item $\ClTerm_{\LPA}(x)$ states that $x$ is (a code of) a closed arithmetical term.
		\item $\ClTermSeq_{\LPA}(x)$ states that $x$ is (a code of) a sequence of closed arithmetical terms.
		\item $\Var(x)$ states that $x$ is (a code of) a first-order variable.
		\item $y=\FV(x)$ states that $y$ is the set of free variables of $x$ (which is either a term or a formula in the language of arithmetic).
		\item $\form_{\LPA}(x)$ states that $x$ is (a code of) an arithmetical formula.
		\item $\form^{\leq 1}_{\LPA}(x)$ states that $x$ is (a code of) an arithmetical formula with at most one free variable.
		\item $\Sent_{\LPA}(x)$ states that $x$ is (a code of) an arithmetical sentence.
		\item $y=\num{x}$ is a binary formula which states that $y$ is (a code of) a numeral denoting the number $x$. 
		\item $y = \val{x}$ states that $x$ is a closed arithmetical term and $y$ is its value. For instance $(\mathbb{N}, S, +, \times) \models \val{\left(\qcr{ S(0) + S(S(0))} \right)} = 3$.
		\item $\Val(\alpha,x)$ states that $\alpha$ is an assignment for $x$, i.e., a finite function whose domain contains all free variables of $s$, where $x$ is either a formula or a term. We will use $\Val(x)$ to denote the set of assignments of $x$ and write $\alpha \in \Val(x)$ instead of $\Val(\alpha,x)$. If $\alpha$ is an assignment for a formula $\phi$, then by $\phi[\alpha]$, we mean a sentence in which $\num{\alpha(v)}$ has been substituted for $v$, for every  $v$ free variable of $\phi$. If $\alpha$ is an assignment for a term $t$, then $t^{\alpha}$ denotes the value of $t$ under this assignment. 
		\item $\beta \sim_v \alpha$ means that $\beta$ and $\alpha$ are assignments, the domain of $\beta$ is $\dom(\alpha) \cup \{v\}$ (which is possibly the same as $\dom (\alpha)$), and the values of $\beta$ are the same as that of $\alpha$, possibly except for $\beta(v)$.
	\end{itemize}
\end{definition}

We will use some conventions to improve readability of the paper. We will write provably functional formulae as if they were function symbols (which we already started doing above). For instance, we will use the expression $\num{x}$ like a term. In particular, we will typically be suppressing formulae describing syntactic operations and simply write the results of these operations. For instance, if $\phi$ and $\psi$ are codes of sentences, then $T(\phi \wedge \psi)$ is an abbreviation for "For all $z$, if $z$ is the conjunction of $\phi$ and $\psi$, then $T(z)$." We will sometimes confuse formulae with sets defined with these formulae, e.g., writing $x \in \form_{\LPA}$ instead of $\form_{\LPA}(x)$.

The notion of \df{syntactic depth} plays an important role in this paper.
\begin{definition} \label{def_synt_depth}
	Let $\phi$ be a formula. By \df{syntactic depth} of $\phi$, we mean the maximal depth of nesting of connectives and quantifiers in $\phi$. We will denote this by $\dpt(\phi)$. By $\dpt(x)$, we will also mean an arithmetical formula representing this function.
\end{definition}

\subsection{Models of arithmetic}

In this paper, we will make extensive use of model-theoretic techniques. All relevant model-theoretic background may be found in \cite{kaye}. Let us discuss some results of particular importance. 

\begin{definition} \label{def_recursive_saturation}
	Let $M$ be any model over a finite language. A set $p$ of formulae is a \df{type} if at most one free variable $v$ and finitely many parametres $a_1, \ldots, a_n$ occur in formulae contained in $p$, and for every finite subset $\phi_1(v), \ldots, \phi_n(v)$ there is an element
	 $a_0$ from $M$ such that $M \models \phi_i(a_0)$ for all $i \leq n$. The type is \df{realised} if there is an element $a \in M$ which satisfies all formulae in $p$. We say that $p$ is \df{recursive} (or computable) if the set of the G\"odel codes of formulae from $p$ is computable.
	  We say that $M$ is \df{recursively saturated} if any recursive type over $M$ is realised in $M$.\footnote{In the name "recursively saturated," there is admittedly slight tension with the current naming conventions where "computable" is the preferred expression, but "computably saturated" sounds extremely awkward.}
\end{definition}

Recursive saturation is of crucial importance due to the following theorem:
\begin{theorem}[Barwise--Schlipf--Ressayre] \label{th_bsr}
	If $M$ is a countable recursively saturated model of $\Th \supset \PA$ and $\Th'$ is a computable theory consistent with the elementary diagram of $M$, then $M$ can be expanded to a model of $\Th'$.
\end{theorem}
Moreover, one can prove that there is always an expansion of $M$ satisfying $\Th'$ which is once again recursively saturated and thus also satisfies the assumptions of the above theorem. This property of countable recursively saturated models of $\PA$ is called \df{chronic resplendence}. 
Another important property of recursively saturated models is that they can be relatively easily classified.
\begin{definition}
	Let $M \models \PA$. By the \df{standard system} of $M$, we mean the family of $X \subseteq \mathbb{N}$ such that $X = A \cap \mathbb{N}$, where $A$ is definable with parametres in $M$. (Here and everywhere hereafter in the article, we identify the initial segment in a model of $\PA$ isomorphic with natural numbers with the $\mathbb{N}$ itself.)
\end{definition}

\begin{theorem}[Paris--Friedman] \label{th_paris_friedman}
	Suppose that $M,N \models \PA$ are countable and recursively saturated. Then $M \simeq N$ iff they satisfy the same sentences and have exactly the same standard systems. 
\end{theorem}
The same result holds if we replace $\PA$ with any other theory in countable language containing $\PA$ and satisfying full induction for the expanded language.
 Another theorem of crucial importance is:
\begin{theorem}[MacDowell--Specker] \label{th_macdowell_specker}
	Let $M$ be a model over a countable language containing the language of arithmetic and suppose that it satisfies full induction scheme for that language. Then there exists an elementary extension $N \succ M$ such that $M$ and $N$ have the same cardinality and for every $a \in N \setminus M$ and every $b \in M$, $N \models a >b$.
\end{theorem}
In fact, if we restrict ourselves to countable models, it is enough to assume that the model satisfies collection. This was proved in \cite{keisler}, Theorem 28. For a general overview of model theory of collection scheme, see \cite{enayatmohsenipour} (where this result occurs as a part of Theorem 1.2 in a more general context of models with a linear order).
\begin{theorem}[Keisler] \label{th_keisler_kolekcja}
	Let $M$ be a countable model over a countable language containing the language of arithmetic. Suppose that $M$ satisfies full collection scheme for that language.  Then there exists an elementary extension $N \succ M$ such that $M$ and $N$ have the same cardinality and for every $a \in N \setminus M$ and every $b \in M$, $N \models a >b$.
\end{theorem}
If $N$ is an extension of $M$ such that every new element in $N$ is greater than all elements in $M$, then $N$ is called an \df{end-extension} of $M$. If $N$ is an end-extension of $M$, this is denoted by $M \subset_e N$ or $M \prec_e N$ if it is, in addition, elementary. In effect, MacDowell--Specker theorem states that any theory in a countable language which extends $\PA$ and proves full induction scheme for its language has a proper elementary end-extension of the same cardinality.

By taking an arbitrary countable model $M$ of $\PA$, taking elementary end-extensions, and taking unions in the limit steps, we can construct a model $M' \succ_e M$ which has cardinality $\aleph_1$, but whose all initial segments are countable. Such models are called \df{$\omega_1$-like models}. In a similar manner, we can define $\kappa$-like models for an arbitrary cardinal $\kappa$.
 Note that if $M$ is a $\kappa$-like model for a regular $\kappa$, then it must have cofinality $\kappa$ which means that every subset $A \subset M$ of cardinality less than $\kappa$ is bounded.

An easy argument shows that if $M \subset_e N \models \PA$ are nonstandard, then $M$ and $N$ have exactly the same standard systems.
 In particular, if $M \preceq_e N$ are countable and recursively saturated, then by Paris--Friedman Theorem (Theorem \ref{th_paris_friedman}), $M \simeq N$. Moreover, this also holds if $M,N$ are models of some countable theory extending $\PA$ which has full induction.
%%%%%%%%% BW: być może referencje do powyższego faktu albo od razu do Parisa--Friedmana w powyższej formie.

\subsection{Truth}

This paper deals with compositional truth theories. Let us now introduce some of them. A systematic treatment can be found in \cite{halbach}. See also \cite{cies_ksiazka}, where the reader can find more information on arithmetic strength of classical compositional truth theories. We will \emph{not} begin with the most canonical example called $\CT^-$ in which the truth predicate $T$ is compositional for arithmetical formulae, but rather with its generalisation which plays a crucial technical role in our paper. 

\begin{definition} \label{def_ct_restr}
	By $\CT^-{\res}X$, we mean a theory in the language of second-order arithmetic containing arithmetical symbols of $\LPA$, a unary predicate $T$ and the membership relation $x \in X$ between first-order elements and sets. To the axioms of $\PA$, we add the following formulae containing a free second-order variable $X$:
	\begin{enumerate}
		\item $\forall x \ \ T(x)\rightarrow x \in \Sent_{\LPA} \wedge \dpt(x) \in X$
		\item $\forall s,t \in \ClTerm_{\LPA} \ \ T(s=t) \equiv (\val{s} = \val{t}).$
		\item $\forall \phi \in \Sent_{\LPA}  \ \ \Big(\dpt(\neg \phi) \in X \rightarrow T\neg \phi \equiv \neg T \phi \Big).$
		\item $\forall \phi, \psi \in \Sent_{\LPA}  \ \ \Big(\dpt(\phi \vee \psi) \in X \rightarrow T (\phi \vee \psi) \equiv T \phi \vee T \psi \Big).$
		\item $ \forall \phi \in \form^{\leq 1}_{\LPA} \ \ \Big(\dpt(\exists v \phi) \in X \wedge \FV(\phi) \subseteq \{v\} \rightarrow T \exists v \phi \equiv \exists x T \phi[\num{x}/v] \Big).$
		\item $\forall \bar{s},\bar{t} \in \ClTermSeq_{\LPA} \forall \phi \in \form_{\PA} \ \ \bigl(\val{\bar{s}} = \val{\bar{t}} \rightarrow T\phi(\bar{t}) \equiv T \phi(\bar{s})\bigr)$.
	\end{enumerate}
\end{definition}

Hence, $\CT^-{\res}X$ states that $T$ is a compositional truth predicate which behaves well on formulae whose \emph{syntactic depth} is in $X$. Moreover, by condition 1. no other formulae are within the range of $T$. Formally, models for $\CT^-{\res} X$ are models for the language $\LPA\cup\{T\}$ with an extra assignment for $X$. We often employ standard model-theoretic conventions and write the interpretation of $X$ in place of the variable $X$, for example $(M,T)\models \CT^-{\res}M$. In practice, we shall also treat $X$ as an additional predicate and write models for $\CT^-{\res}X$ in the form $(M,T,X)$. We will also be writing $(M,T) \models \CT^- \res X$ for $X \subset M$ with the obvious meaning. We will essentially use the notation $\CT^- \res X$ in one context: when $X$ is a  nonstandard initial segment (possibly with the largest element). 

By $\CT {\res}X$, we mean $\CT^-{\res}X$ along with the full induction scheme for the extended language (i.e. formulae which may use the predicate $T$ and the free variable $X$).
By $\CT^-$ we mean  a theory of an unrestricted compositional truth predicate (i.e., we do not restrict axioms of $\CT^- \res X$ to formulae from the set $X$). By $\CT$ we mean $\CT^-$ with full induction.

In our proof, we will need to impose an unpleasantly technical regularity condition on the truth predicates. Essentially, we want to consider truth predicates which only see syntactic trees of considered formulae and the values of terms which we plug in rather than specific terms and specific variables over which we quantify. 

First, we introduce the notion of the structural template.
\begin{definition} \label{def_structural_template}
	 If $\phi \in \form_{\LPA}$, we say that $\widehat{\phi}$ is its \df{structural template} iff 
	\begin{itemize}
	  	\item No constant symbol occurs in $\widehat{\phi}$.
		\item No free variable occurs in $\widehat{\phi}$ twice.
		\item For every term $t$ occurring in $\widehat{\phi}$, if all variables in $t$ are free, then $t$ is a free variable.
		\item No variable occurs in $\widehat{\phi}$ both as a bounded and as a free variable. 
		\item The formula $\phi$ can be obtained from $\widehat{\phi}$ by renaming bounded variables and substituting terms for free variables in such a way that no variable appearing in those terms becomes bounded.
		\item $\widehat{\phi}$ is the smallest formula with those properties.
	\end{itemize}
We say that formulae $\phi, \psi$ are \df{structurally similar}, $\phi \sim \psi$ iff $\widehat{\phi} = \widehat{\psi}.$
\end{definition}
\begin{example}
	If $\phi = \forall x \exists x \forall z \ (x+ S0) \times (((y \times S0) + 0) + (y +z) ) = (z \times (x + S0)) \times S y $, then
	\begin{displaymath}
	\widehat{\phi} = \forall w_1 \exists w_1 \forall w_2 \ (w_1+ v_0) \times (v_1 + (v_2 + w_2)) = (w_2 \times (w_1 + v_3)) \times v_4,
	\end{displaymath}
	where $w_i, v_i$ are chosen so as to minimise the formula. 
\end{example}
\begin{example}
	\begin{enumerate}
		\item 	The formulae $\phi_1 = \exists x \forall x \ x+y =0$ and $\phi_2 = \exists y \forall y \ y+x = z \times SSS0$ are structurally similar.
		\item The formulae $\psi_1 = \exists x \forall y \ x+y = 0$ and $ \psi_2 = \exists y \forall y \  x+y= 0$ are not structurally similar, because the first quantifies over two distinct variables and the second does not. 
		\item The formulae $\eta_1 = \exists x \forall y \ x+y = 0$ and $ \eta_2 = \exists x \forall y \  x+y= y$ are not structurally similar because in the first one, the universally quantified variable occurs only once under the scope of the quantifier. 
	\end{enumerate}
\end{example}
By induction on the complexity of formulae, one can easily check that any formula has its structural template. By minimality, it is unique. Now, we are able to define our desired notion.

\begin{definition} \label{def_structural_equivalence}
	Let $\phi, \psi$ be two sentences. We say that they are \df{structurally equivalent} iff they are structurally similar and there exist two formulae $\phi^*, \psi^*$ which differ from $\widehat{\phi}$ by renaming bounded variables, and sequences $\bar{s}, \bar{t}$ of closed terms such that $\bar{\val{s}} = \bar{\val{t}}$ (i.e., all the terms in the sequence have the same values) for which $\phi = \phi^*(\bar{t})$ and $\psi = \psi^*(\bar{s})$.
	
	We denote this relation with $\phi \approx \psi$.
\end{definition}

\begin{example}
	Suppose that $\phi = \exists x \ x+S0 = 0 + S0$ and $\psi = \exists y \ y + (S0 + 0)\times S0 = S0.$ Then $\phi \approx \psi.$
\end{example}

\begin{definition} \label{def_structural regularity}
	By the \df{structural regularity principle} ($\StrReg$) we mean the following axiom:
	\begin{displaymath}
	\forall \phi, \psi \in \Sent_{\LPA} \Big(\phi \approx \psi \rightarrow T \phi \equiv T \psi \Big).
	\end{displaymath}
\end{definition}
In what follows, we will essentially work with the theory $\CT^- + \StrReg$. One of the fundamental results in the theory of truth states that $\CT^-$, the theory of compositional truth predicate, does not prove any new arithmetical theorems:
\begin{theorem}[Kotlarski--Krajewski--Lachlan] \label{th_kkl}
	$\CT^-$ is conservative over $\PA$.
\end{theorem}
On the other hand, as proved in \cite{lachlan}, the presence of $\CT^-$ has nontrivial consequences on the model-theoretic side.
\begin{theorem}[Lachlan] \label{th_lachlan}
	If $(M,T) \models \CT^-$, then $M$ is recursively saturated. Moreover, the same holds for $\CT^-{\res}[0,c]$ for every nonstandard $c\in M$.
\end{theorem}

By a simple compactness argument, using arithmetical partial truth predicates (say, for $\Sigma_n$ classes, as every formula in $\Sigma_n$ has depth $\leq n$), one can show that the theory of the restricted compositional truth predicate is still conservative also when we consider the fully inductive variant.
\begin{theorem}
	The theory saying "$I$ is a nonstandard initial segment and $\CT{\res}I$"  is conservative over $\PA$.
\end{theorem}
Recall that $\CT{\res}I$ is simply a theory of truth predicate which satisfies full induction and compositional conditions for formulae whose \emph{depth} is in an unspecified initial segment $I$. The obvious common strengthening of the two theories, i.e. $\CT$, is much stronger than $\PA$. For example, arguing by induction on the length of proofs, we can easily see that $\CT$ proves the consistency of $\PA$.
\begin{theorem}
	$\CT$ is not conservative over $\PA$.
\end{theorem}
More generally, using essentially the same proof, we can show the following
\begin{theorem}\label{ubi prawda ibi konsystencja}
	If $(M,T)\models \CT\res[0,c]$, then every proof of contradiction of $\PA$ in $M$ contains a formula of depth $>c$. 
\end{theorem}
Interestingly, we can conservatively add to $\CT^-$ some specific form of the induction scheme. By \df{internal induction} we mean the axiom:
\begin{equation*} \label{eq_int} \tag{$\INT$}
\forall \phi \in \form^{\leq 1}_{\LPA} \Bigl( T\phi(\num{0}) \wedge \forall x \ \ \left(T \phi(\num{x}) \rightarrow T \phi(\num{x+1}) \right) \rightarrow \forall x T \phi(\num{x}) \Bigr).
\end{equation*}
It essentially states that any set defined with a (possibly nonstandard) formula under the truth predicate satisfies the induction scheme. As we have already mentioned, the following holds:\footnote{The result is announced, but not really proved, in \cite{kkl} and \cite{enayatvisser2}. A complete proof occurs in \cite{leigh}.}
\begin{theorem}[Kotlarski-Krajewski-Lachlan] \label{th_ctminus_int_conservative}
	$\CT^- + \INT$ is conservative over $\PA$.
\end{theorem}
The same result holds if we consider a theory $\CT^- + \INT + \StrReg$. We will not show it, but we will discuss the proof in the Appendix, since it is a slight modification of the proof of Lemma \ref{lem_enayat_visser_z_kolekcja}.
\begin{theorem} \label{th_ctminus_plus_int_plus_strreg_konserwatywne}
	$\CT^- + \StrReg + \INT$ is conservative over $\PA$.
\end{theorem}

\section{The main result}
In this section, we will prove our main theorem. As we have written in the introduction, Richard Kaye asked whether $\CT^-$ with full collection scheme, but without any induction, is conservative over $\PA$. We provide a partial answer to this question. Let us consider the scheme of \df{local collection}, $\LocColl$, which consists of the following formulae:
\begin{displaymath}
\forall c \Bigl(\forall x< a \exists y \ \ \phi[T_c/T](x,y) \rightarrow \exists b \forall x < a  \exists y<b \ \ \phi[T_c/T](x,y) \Bigr) ,
\end{displaymath}
where $\phi$ is an arbitrary formula of $\LPA$ extended with a truth predicate and $T_c(x)$ abbreviates $T(x) \wedge \dpt(x) \leq c$. Local collection expresses that any such restriction of the truth predicate, $T_c$, satisfies full collection scheme.

\begin{theorem} \label{tw_konserwatywnosc_lokalnej_kolekcji}
	$\CT^- + \LocColl$ is conservative over $\PA$.
\end{theorem}
The proof of the main result relies on the following simple and well-known  observation.

\begin{proposition} \label{stw_kappa_like_daje_kolekcje}
	Let $M \models \PA$ be a $\kappa$-like model for some regular cardinal $\kappa$. Let $T \subset M$. Then $(M,T)$ satisfies full collection scheme.
\end{proposition}
\begin{proof}
	For any function $f: M \to M$ and any $a \in M$, the image of the initial segment $f [[0,a]]$ has less then $\kappa$ elements and thus it is bounded. This immediately implies that the collection scheme holds in $M$ expanded with an arbitrary predicate.  
\end{proof}

The observation suggests one possible strategy of the proof that collection for the truth predicate is conservative over $\PA$. If up to elementary equivalence,  for any countable $M \models \PA$, we can find an $\omega_1$-like elementary end-extension $M'$ with a truth predicate $T$, then this predicate $T$ must automatically satisfy the collection scheme, which in turn implies that collection is conservative over $\PA$.\footnote{This strategy of proof was explicitly suggested by Richard Kaye. However, as we already mentioned, the most obvious strategy of building $\omega_1$-like models of $\CT^-$ simply does not work, due to the fact that not every model of $\CT^-$ has an end-extension.} Basing on this approach, we are able to show conservativity of local collection.

 We shall rely heavily on the construction of the \df{disintegration} of a truth predicate. To better understand it, observe that in the arithmetical context, the truth predicate canonically determines a satisfaction relation $S_T(x,y)$ via the definition:
\[S_T(\phi,\alpha) := \form_{\LPA}(\phi) \wedge \Val(\alpha,\phi) \wedge T(\phi[\alpha]).\]
Then, the disintegration of a truth predicate is simply an infinite family of projections of $S_T$ along the singleton sets $\{\phi\}$, for each formula $\phi$ in the considered model.

\begin{definition}\label{def_dezintegracja}
	Let $(M,T,X)\models \CT^-\res X$. The \df{disintegration} of $T$ is a family of predicates $\{X_{T\phi}\}_{\phi\in M}$ which are interpreted in $(M,T, X)$ by the condition 
	\[\alpha \in X_{T\phi} \textnormal{ iff } \Val(\alpha,\phi) \wedge T(\phi[\alpha]).\]
\end{definition}  

The idea of disintegration is that we expand $M$ with all relations (possibly of nonstandard arity) which are arithmetically definable with possibly nonstandard formulae using the predicate $T$. Let us notice that this newly obtained structure corresponds to the original one in a very direct way. Namely, the following equivalence holds for any $\eta \in \form_{\LPA}(M)$:
\begin{displaymath}
( M,T, X_{T\phi})_{\phi \in M} \models \forall \alpha \in \Val(\eta) \ \  X_{T\eta}(\alpha)  \equiv T \eta[\alpha]. 
\end{displaymath}

 Let us proceed to the lemma which is the technical core of our proof. Note that, in view of Theorem \ref{ubi prawda ibi konsystencja}, already here a dramatic difference from the induction scheme becomes apparent. In what follows, we will denote the \df{elementary diagram} of a model $M$ by $\ElDiag(M)$.

\begin{lemma} \label{lem_ct_minus_ograniczone_z_kolekcja_jest_konserwatywne}
	For any $M \models \PA$ and any $c \in M$, the theory $\ElDiag(M) + \CT^- \res [0,c] + \Coll$  is consistent.
\end{lemma} 
Recall that $\CT^- \res [0,c]$ is the theory of a compositional truth predicate for formulae of depth at most $c$.

Since, we will need to use the above lemma iteratively, we will need its strengthening which is proved in almost the same way, so we will only present the proof of the following version:

\begin{lemma} \label{lem_iteracja_lematu_o_kolekcji}
	Let $M \models \PA$ be a countable recursively saturated model and let $c<d$ be any two elements. Suppose that $(M,T) \models \CT^- \res [0,c] + \Coll + \StrReg$ is recursively saturated in the expanded language. Then there exists $T' \supset T$ such that $(M,T') \models \CT^- \res [0,d] + \Coll + \StrReg$.
\end{lemma}

In \cite{enayatvisser2}, the conservativity of $\CT^-$ has been demonstrated with an elegant, model-theoretic reasoning. The proof presented there allows numerous modifications in order to obtain finer results. We will make use of one such strengthening. We will use it in the proof of Lemma \ref{lem_iteracja_lematu_o_kolekcji}.

\begin{lemma} \label{lem_enayat_visser_z_kolekcja}
	Suppose that $(M,T,I) \models \CT^-{\res} I+ \Coll + \StrReg$ is a countable model recursively saturated in the extended language with $I$ an initial segment, possibly empty. Then there exists $T' \supset T$ such that $(M,T') \models \CT^- + \StrReg$ and, moreover, the model $(M,T,X_{T'\phi})_{\phi \in M}$ satisfies full collection scheme, where the family $\{X_{T'\phi}\}_{\phi\in M}$ is the disintegration of $T'$.
\end{lemma}

Notice that the crucial point of the above lemma is that the predicates $X_{\phi}$ satisfy collection \emph{jointly} with the original predicate $T$.

The lemma is proved by combining a resplendence argument and the Enayat--Visser construction.  The details are standard and are given in the Appendix. We now turn to the proof of Lemma \ref{lem_iteracja_lematu_o_kolekcji}.

\begin{proof}[Proof of Lemma \ref{lem_iteracja_lematu_o_kolekcji}]
	Let $(M,T), c, d$ be as in the assumption. Using Lemma \ref{lem_enayat_visser_z_kolekcja}, we construct a model $(M,T^*) \models \CT^- + \StrReg$ such that $T^* \supset T$ and the predicates $\{X_{T^*\phi}\}_{\phi \in M}$ satisfy full collection jointly with $T$. 
	
	We will show that $M$ has an elementary end extension $M'$
	such that for some $T'\subseteq M'$ extending $T^*$, $(M',T') \models \CT^-{\res}M + \Coll+ \StrReg$. 
	In particular, it follows that for any $c \in M$,
	\begin{displaymath}
	(M',T') \models \ElDiag(M) + \CT^- \res [0,c] + \Coll +  \StrReg.
	\end{displaymath}
	By resplendence, this will conclude the proof. (Note that $T'$ which we construct in the proof is not literally the same as $T'$ satisfying the conclusion of the lemma, but we would like to avoid employing excessively heavy notation.)

	Since $(M,T,X_{T^*\phi})_{\phi \in M}$ 	satisfies full collection scheme, by Keisler's Theorem \ref{th_keisler_kolekcja}, it has an elementary end extension. By taking an $\omega_1$-chain of such elementary end-extensions, we obtain a model  $(M',X'_{T^*\phi})_{\phi \in M}$ elementarily extending $(M,X_{T^*\phi})_{\phi \in M}$, where $M'$ is an $\omega_1$-like model. 
	
	Now, let
	\begin{multline*}
	T_0' =  \Bigl\{ \phi(t_1, \ldots, t_e) \in \Sent_{\LPA}(M') \mid \phi \in \form_{\LPA}(M), \tuple{t_1, \ldots, t_e} \in \ClTermSeq_{\LPA}(M'), \\
	(M',X'_{T^*\phi})_{\phi \in M} \models X_{T^*\phi}(\tuple{\val{t_1}, \ldots, \val{t_e}}) \Bigr\},
	\end{multline*}
	where we conflate a sequence of values and a corresponding assignment. Notice that we do not assume that the sequence $\tuple{\val{t_1}, \ldots, \val{t_e}}$ has standard length or standard values. Let finally:
	\begin{multline*}
	T' =  \Bigl\{ \phi \in  \Sent_{\LPA}(M') \mid \exists \psi \in \form_{\LPA}(M) \exists \tuple{t_1, \ldots, t_e} \in \ClTermSeq_{\LPA}(M') \\
	\psi \approx \phi(t_1, \ldots, t_e) \Bigr\}.
	\end{multline*}
	
	We claim that $(M',T') \models \CT^-{\res} M + \Coll+\INT + \StrReg$. This model satisfies collection scheme by Proposition \ref{stw_kappa_like_daje_kolekcje}, since $M'$ is $\omega_1$-like. Notice that $T_0'$ was defined only for sentences obtained by substituting terms into formulae from $M$, whereas we want to make sure that it is defined on formulae whose \emph{depth} is in $M$. However, since $M'$ is an end-extension of $M$ for any formula whose depth is in $M$, its syntactic template is in $M$ as well. 
	
	We first check that the compositional conditions are satisfied for $T'$ and the sentences $\phi(t_1, \ldots, t_e)$ for $\phi \in \form_{\LPA}(M)$ by cases which depend on the syntactic shape of a formula $\phi$. For example, let $\phi(t_1, \ldots, t_e) = \exists v \psi(v,t_1, \ldots,t_e)$ where $\psi \in \form_{\LPA}(M)$ and $\tuple{t_1, \ldots, t_e} \in M'$. The equivalence
	\begin{displaymath}
	\forall \alpha \in \Val(\phi) \ \ \Big( X_{T^*\phi}(\alpha) \equiv \exists \beta \sim_v \alpha \ \ X_{T^*\psi}(\beta) \Big)
	\end{displaymath} 
	holds in $(M,X_{T^*\phi})_{\phi \in M}$, since $T^*$ satisfies compositional conditions. Therefore it must hold in  $(M',X'_{T^*\phi})_{\phi \in M}$ by elementarity. So by definition $T'$ satisfies the compositional condition for the quantifier for the formula $\phi$. The other cases are analogous. 
	
	The compositional conditions are satisfied for other sentences with depth in $M$ as well. Take any formula $\phi \in M'$ such that $\dpt(\phi) \in M$. First observe that if $\phi \sim \psi$  and $\psi \in M$, then $\widehat{\phi}  \in M$, since by elementarity $\widehat{\psi} \in M$ and these two are equal. Then we check that $T'$ is compositional by case distinction depending on the main connective or quantifier in $\phi$. 
	
	For instance, suppose that  $\phi = \exists v \eta$, $T'\phi$ holds, and $\phi \approx \psi = (\exists w \xi) \in M$ such that $T'\psi$ holds. Then by compositionality of $T_0'$, there exists $x \in M'$ such that $T' \xi(\num{x})$ holds. Now, since $\eta(\num{x}) \approx \xi(\num{x})$, by definition $T'\eta(\num{x})$ holds as well. An analogous reasoning shows that if $T' \eta(\num{x})$ holds for some $x \in M'$, then $T' \phi$ holds. The argument for disjunction is similar.

	The argument for negation is the only place where we use $\StrReg$. Namely, suppose that $T' \neg \phi$ holds for some $\phi \in M'$. We want to show that $T'\phi$ does not hold. Suppose otherwise. By definition of $T'$, there exists $ \psi \approx \phi$ such that $\psi = \psi^*(t_1,\ldots,t_n)$ for some $\psi^* \in \form_{\LPA}(M)$ and $T_0'\neg \psi$ holds. By compositionality, $T'_0 \psi$ does not hold. Now, by $\StrReg$ $T_0' \eta$ cannot hold for any $\eta \approx \psi$. In particular, it cannot hold for any $\eta \approx \phi$. The other implication for the compositionality of negation can be proved with a simple argument similar to the argument for the existential quantifier. 
	
	It follows immediately by the construction that $T'$ satisfies the structural regularity property $\StrReg$.
	\begin{comment}
	Finally, the internal induction, $\INT$, holds in $(M',T')$ since if $\dpt(\phi)\notin M$, then $(M',T')\models \neg T'\phi(\num{0})$, hence the axiom of induction for $\phi$ is vacuously true. If $\dpt(\phi)\in M$, then $X'_{T^*\phi}$ satisfies the induction axiom by elementarity and the internal induction for all formulae with depth in $M$ follows.
	\end{comment}
\end{proof}

Now, we are ready to prove our theorem.
\begin{proof}[Proof of Theorem \ref{tw_konserwatywnosc_lokalnej_kolekcji}]
	Let $M \models \PA$ be any countable recursively saturated model. Fix any sequence $(a_n)_{n \in \omega}$ cofinal in $M$. Using Lemma \ref{lem_ct_minus_ograniczone_z_kolekcja_jest_konserwatywne} in the initial step and Lemma \ref{lem_iteracja_lematu_o_kolekcji} and chronic resplendence in the induction step, we construct a sequence of predicates $T_n$ such that
	\begin{displaymath}
	(M,T_n) \models \CT^- \res [0,a_n] + \Coll,
	\end{displaymath}
	 the constructed models are recursively saturated in the expanded language.
	
	Finally, we set $T:= \bigcup_{n \in \omega} T_n$. Then we readily check that $(M,T) \models \CT^- + \LocColl$. Since $M$ was arbitrary, this concludes the proof.
\end{proof}

\section{Local induction}

	As we have already noted, the behaviour of local collection is in stark contrast to the behaviour of local induction which is its natural analogue for the induction scheme. More precisely, let us define the instances of local induction, $\LocInd$, as follows:
	\begin{displaymath}
	\forall c \ \ \Bigl(\phi[T_c/T](0) \wedge \forall x \bigl(\phi[T_c/T](x) \rightarrow \phi[T_c/T](x+1) \bigr) \rightarrow \forall x \phi[T_c/T](x) \Bigr),
	\end{displaymath}	
	where $\phi$ is an arbitrary formula in the language $\LPA$ extended with a truth predicate and $T_c(x)$ is an abbreviation for $T(x) \wedge \dpt (x) \leq c$. In other words, 
	\begin{equation}\label{equat_loc_ind}\tag{\LocInd}
	(M,T)\models \CT^- + \LocInd \textnormal{ iff } \forall c\in M, (M,T_c)\models \CT{\res}[0,c],
	\end{equation}
	so local induction scheme expresses that any restricted truth predicate $T_c$ satisfies full induction.

	One can easily observe that local induction is not conservative over $\PA$, since it proves the consistency of $\PA$. Indeed, by composing Theorem \ref{ubi prawda ibi konsystencja} and the above condition \eqref{equat_loc_ind} one gets that for every $c\in M$, every proof of $0=1$ in $\PA$ contains a formula of complexity $>c$. Let us briefly recall the whole argument: take any proof $d$ in $\PA$, say, in Hilbert calculus. There exists $c$ such that all sentences occurring in that proof have complexity smaller than $c$. Take the restricted predicate $T_c$ and show, using local induction, that every sentence in that proof is true. Consequently, the conclusion of the proof has to be true, and thus it cannot be of the form "$0 \neq 0$."

 	The above proof essentially shows that in $\CT^- + \LocInd$, we can show the following principle of \df{global reflection}: 
 	\begin{equation}\label{GR}\tag{GR}
 	\forall \phi \in \Sent_{\LPA} \ \ \Pr_{\PA}(\phi) \rightarrow T \phi,
 	\end{equation}
 	where $\Pr_{\PA}(x)$ is the canonical provability predicate for $\PA$. In order to prove global reflection, we fix any $\phi$ which is provable in $\PA$, we fix any proof of $\phi$ and take any $b$ such that all formulae in the proof have depth smaller than $b$. Then we take the restriction $T_b$ and show by induction on the length of derivation that all formulae in the proof are true under all assignments.
 	
 	As shown by Kotlarski in \cite{kotlarski},
 	$\CT^-$ with global reflection proves $\Delta_0$-induction for the truth predicate. His argument was later refined in two ways by Cezary Cieśliński: firstly, in \cite{cies} it was shown that reflection over first order logic (i.e. \eqref{GR} with $\Pr_{\PA}$ changed to $\Pr_{\emptyset}$) is sufficient to prove $\Delta_0$ induction.
 	 Secondly, in \cite{cieslinskict0} it was shown that the closure under propositional logic principle, i.e. the sentence
 	\[\forall \phi\ \ \Pr_{\textnormal{Prop}}^T(\phi) \rightarrow T(\phi),\]
 	where $\Pr_{\textnormal{Prop}}^T(\phi)$ expresses that $\phi$ is provable from true premises in pure propositional calculus, is enough to yield bounded induction.\footnote{We note, however, that the last principle expresses closure of the set of true sentences under a logical reasoning. Thus we potentially require something more than in the previous two reflection principles.}

	Kotlarski in \cite{kotlarski} characterised the arithmetical strength of global reflection in terms of the following family of theories:
	\begin{align*}
	\Th_0 &= \PA\\
	\Th_{n+1} &= \set{\forall x \phi(x)}{\phi(x)\in\LPA, \forall k \in \omega \ \ \Th_{n}\vdash \phi(\num{k})}
	\end{align*}
	\begin{theorem}[Kotlarski]
		$\CT^- + \eqref{GR}$ is arithmetically conservative over $\PA + \set{\Con(\Th_n)}{n\in\omega}$.
	\end{theorem}
	%%as $\omega$-many iterations of the uniform reflection scheme over $\PA$. 
	An easy argument shows that the above arithmetical theory is equivalent to $\omega$-many iterations of the uniform reflection principle over $\PA$. Details concerning the inclusion of Kotlarski's theory in the iterations of reflection can be found in the paper \cite{smorynski} and in the second author's PhD Thesis, \cite{lelyk_thesis}.  
	
	It turns out that the content of $\LocInd$ can be characterised in a very precise manner. We have just shown that it implies global reflection \ref{GR}. It turns out that $\LocInd$ is exactly equivalent to \ref{GR}. 
	
	\begin{fact} \label{fakt_locind_equivalent_gr}
		$\CT^- + \LocInd$ is equivalent to $\CT^- +$ \textnormal{\ref{GR}}.
	\end{fact}

	Moreover, it was shown in the second author's PhD thesis \cite{lelyk_thesis} that $\CT^- +$ \ref{GR} is equivalent to $\CT_0$, the compositional truth theory $\CT^-$ extended with bounded induction for the full language which immediately allows us to obtain an equivalent characterisation.
	
	\begin{fact} \label{fak_locind_equivalent_ct0}
		$\CT^- + \LocInd$ is equivalent to $\CT_0$. 
	\end{fact}
	
	It is relatively straightforward to show that $\CT^- + $ \ref{GR} proves internal induction, $\INT$. Let $\ind(\phi)$ abbreviate the axiom of induction for a formula $\phi\in \form_{\Lang_{\PA}}^{\leq 1}$, i.e. the sentence
	\[\phi(0)\wedge \forall x \bigl(\phi(x)\rightarrow \phi(x+1)\bigr)\rightarrow \forall x \phi(x).\]
	We work in $\CT^- + $ \ref{GR}. Since for every $\phi\in \form_{\Lang_{\PA}}^{\leq 1}$ we have $\Pr_{\PA}(\ind(\phi))$, by \ref{GR} it follows that $T(\ind(\phi))$ holds. By compositional axioms and extensionality we obtain the sentence
	\[T(\phi(0))\wedge \forall x\bigl(T(\phi(\num{x}))\rightarrow T(\phi(\num{x+1}))\bigr)\rightarrow \forall x T(\phi(\num{x})).\]
	%%%%%%%%%% BW 9.06 Mateusz to zmienił. Spytać o rezmianę.

 	It is a classical fact of first-order arithmetic that there exist partial $\Sigma_n$-truth predicates. More precisely, the following holds provably in $\PA$:\footnote{For a detailed discussion of arithmetical truth predicates, see \cite{hajekpudlak}, Chapter I, Section 1(d), pp.50--61.}
 	\begin{theorem} \label{th_arytmetyczne_predykaty_prawdy}
 		For every $n$, there exists a formula $\Tr_n$ such that for every sentence $\phi$ with $\dpt(\phi) \leq n$ (in fact, for $\phi \in \Sigma_n$), the following equivalence holds:
 		\begin{displaymath}
 		\Tr_n(\num{\phi}) \equiv \phi.
 		\end{displaymath}   
 	\end{theorem}
This theorem formalises in $\PA$, hence  we have:
	\begin{displaymath}	\forall c \forall \phi \in \Sent_{\LPA} \Big(\dpt(\phi) \leq c \rightarrow \Pr_{\PA}(\Tr_c(\num{\phi}) \equiv \phi) \Big).
	\end{displaymath}
	 By \ref{GR} and the compositional axioms we obtain:
	\begin{displaymath}
	\forall c \forall \phi \in \Sent_{\LPA} \Big(\dpt(\phi) \leq c \rightarrow (T\Tr_c(\num{\phi}) \equiv T\phi) \Big).
	\end{displaymath}
	Let $\Theta_c(x)$ be defined as $T\Tr_c(\num{x}) \wedge x \in \Sent_{\LPA} \wedge \dpt(x) \leq c$. Fix any instance of the induction scheme containing the truth predicate $T_c(x)$:
		\begin{displaymath}
	\phi[T_c](0) \wedge \forall x \bigl(\phi[T_c](x) \rightarrow \phi[T_c](x+1) \bigr) \rightarrow \forall x \phi[T_c](x).
	\end{displaymath}
	Since $T_c$ and $\Theta_c$ are equivalent by the above considerations, the displayed sentence is equivalent to:
	\begin{displaymath}
	\phi[\Theta_c](0) \wedge \forall x \bigl(\phi[\Theta_c](x) \rightarrow \phi[\Theta_c](x+1) \bigr) \rightarrow \forall x \phi[\Theta_c](x).
	\end{displaymath}
	But, by applying compositional axioms for the full truth predicate $T$, we can "pull it up" from $\Theta_c$ to the top of the formula $\phi$, thus obtaining:
	\begin{displaymath}
	T\phi[\Tr'_c](\num{0}) \wedge \forall x \bigl(T\phi[\Tr'_c](\num{x}) \rightarrow T\phi[\Tr'_c](\num{x+1}) \bigr) \rightarrow \forall x T\phi[\Tr'_c](\num{x}),
	\end{displaymath}
	where $\Tr'_c(x)$ is the formula $\Tr_c(x)  \wedge x \in \Sent_{\LPA} \wedge \dpt(x) \leq c$. The last formula is however an instance of the internal induction axiom and thus is provable in $\CT^- +$ \ref{GR}. This shows that $\LocInd$ holds in $\CT^- +$ \ref{GR}.

	\section{The strength of $\BSigma_n(T)$}
	
	Let us now consider a question whether adding a little bit of collection to $\CT_0$ increases the arithmetical strength of the latter theory. Let us denote
	\[\BSigma_n(T):= \CT_0 + \SigmanColl,\] 
	where $\SigmanColl$ is the restriction of  full collection scheme to $\Sigma_n$ formulae in the expanded language. It is easy to observe that, as in the purely arithmetical setting, we have
	\[\BSigma_{n+1}(T)\vdash \CT_n,\]
	and $\CT_{n+1}\vdash \Con (\CT_n)$, hence already $\BSigma_2(T)$ is arithmetically non-conservative over $\CT_0$. What is left is the case of $\Sigma_1$ collection: we shall show that it is $\Pi_2$-conservative over $\CT_0$ (over the full language with the truth predicate), which implies that $\BSigma_1(T)$ is arithmetically conservative over $\CT_0$ as for every arithmetical sentence $\phi$, $T(\phi)$ is an atomic  sentence of the expanded language equivalent to $\phi$ (provably in $\CT^-$). More generally, the situation for fragments of $\CT$ parallels the one well known from fragments of $\PA$:
	
	\begin{theorem}\label{th_besigmaen}
		For every $n\geq 0$, $\BSigma_{n+1}(T)$ is $\Pi_{n+2}$ conservative over $\CT_n$ in the extended language. In particular for all $n$, $\BSigma_{n+1}(T)$ is arithmetically conservative over $\CT_n$.
	\end{theorem}
	
	 Although the proof follows essentially by the same pattern of reasoning as in the classical Paris--Friedmann result (see \cite{hajekpudlak}, Theorem 1.61, Chapter IV or \cite{kaye}, Corollary 10.9), one detail has to be taken care of. It is the content of the following lemma. Let us recall that if we have a model $M$ and a set $I\subseteq M$, then 
	\begin{center}
	$\sup_M(I) := \set{x\in M}{\exists b\in I\ \ M\models x<b}.$
	\end{center}
	If $M$ is a model of $\PA^-$ and $I$ is closed under multiplication, then $\sup_M(I)$ is a substructure of $M$. If additionally $M\models \CT^-$, then we can naturally view $\sup_M(I)$ as a substructure of $M$.
	\begin{lemma} \label{lem_odcinek_poczatkowy_generuje_ctminus}
		Suppose $M\preceq N$ are models of $\CT_0$. Then  $\sup_N(M)\models \CT^-$. Consequently, $\sup_N(M) \models \CT_0$.
	\end{lemma}
	\begin{proof}
		The only problematic issue is whether $\sup_N(M)$ satisfies the compositional axiom for the existential quantification, i.e. 
		\[  \forall \phi \in \form^{\leq 1}_{\PA} \ \ \Big(\FV(\phi) \subseteq \{v\} \rightarrow T \exists v \phi \equiv \exists x \ T \phi(\num{x}) \Big),\]
Fix $\phi$ and $v$ as above and put $I = \sup_N(M)$. Given that $I\subseteq N$, the non-obvious part is whether $I$ validates the implication $T \exists v \phi \rightarrow \exists x T \phi(\num{x})$. Working in $I$, assume $T(\exists v \phi)$. Let $d\in M$ be  greater than $\exists v \phi$ (as an element of $N$). Consider the following sentence
	\[\exists c\forall v <d\forall\psi <d \ \ \biggl( \bigl(\Var(v) \wedge \form_{\LPA}^{\leq 1} (\psi) \wedge T\exists v \psi\bigr) \rightarrow \exists x<c \ T\psi(\num{x})\biggr).\]
	The above is true in $M$, since it is equivalent to 
	\[\exists c\forall v <d\forall\psi <d \ \ \biggl( \bigl(\Var(v) \wedge \form_{\LPA}^{\leq 1} (\psi) \wedge T_{d+1}\exists v \psi\bigr) \rightarrow \exists x<c \ T_{d+1}\psi(\num{x})\biggr)\]
	which is an instance of the strong collection scheme for $T_{d+1}$ and each restriction of $T$ is fully inductive by $\LocInd$ which is equivalent to $\CT_0$ by Fact \ref{fak_locind_equivalent_ct0}. Fix $c\in M$ witnessing the existential quantifier. By elementarity 
	\[\forall v <d\forall\psi <d \ \ \biggl( \bigl(\Var(v) \wedge \form_{\LPA}^{\leq 1} (\psi) \wedge T\exists v \psi\bigr) \rightarrow \exists x<c \ T\psi(\num{x})\biggr)\]
	is true in $N$, but as it is a $\Delta_0(T)$ sentence with parameters from $I$, it holds in the latter model as well (by definition $I \subseteq_e N$). This ends the proof since $d$ majorizes both $v$ and $\phi$ in $\sup_N(M)$.
	\end{proof}

	The rest of the proof of Theorem \ref{th_besigmaen} can be carried out as in the case of arithmetical theories. We will use the following lemma as the key ingredient.
	
	\begin{lemma} \label{lem_rozszerzenie_indukcja_daje_kolekcje}
		Let $I \preceq_n N$ be two models in a language extending $\LPA$, possibly with additional predicates, where $n \in \mathbb{N}$ (we allow $n=0$). If $N$ satisfies $\Sigma_n$-induction, then $I$ satisfies $\Sigma_{n+1}$-collection. 
	\end{lemma}
	The proof of this fact for the language of arithmetic is given e.g. in \cite{kaye}, Proposition 10.5.  It transfers to languages extending $\LPA$ after obvious modifications.
	\begin{proof}[Proof of Theorem \ref{th_besigmaen}]
	 In order to prove the theorem, it is enough to show that if $\phi \in \Sigma_{n+2}$ is satisfied in some model of $\CT_n$, then it is satisfied in some model of $\BSigma_{n+1}(T)$.
			
	So fix a model $M\models \CT_n + \phi$, where $\phi = \exists x \forall y \psi(x,y)$, $\psi\in \Sigma_n$. Let $N$ be an elementary extension of $M$ such that $\sup_N(M)\neq N$. Put $I = \sup_N(M)$. By elementarity, there exists $c\in M$ such that
	\[N\models \forall y \psi(c,y),\]
	If $\psi(x,y)$ is $\Delta_0$, then it automatically follows that the same is true in $I$, which, by Lemma \ref{lem_odcinek_poczatkowy_generuje_ctminus} is a model of $\CT_0$ and thus by Lemma \ref{lem_rozszerzenie_indukcja_daje_kolekcje}, a model of $\BSigma_1(T)$. This concludes the proof for $n=0$.

	 For greater $n$'s we have to show that $I\prec_n N$. This will conclude our argument since then again we obtain that $I \models \forall y \psi(c,y)$ by elementarity and that $I \models \BSigma_{n+1}(T)$ by Lemmata \ref{lem_odcinek_poczatkowy_generuje_ctminus} and \ref{lem_rozszerzenie_indukcja_daje_kolekcje}.
	
	We show $\Sigma_n$-elementarity: employing the Tarski--Vaught test, it is sufficient to show that for all $\theta(x,y)\in\Sigma_n$ and all $b\in I$
	\[N\models \exists x \theta(x,b) \Rightarrow \exists d\in I\ \  N\models \theta(d,b).\] 
	So fix $\theta(x,y)\in\Sigma_n$, $b\in I$ and assume $N\models \exists x \theta(x,b)$. Fix $e\in M$ such that $N\models b<e$. Since $M\models \CT_n$ we have an $f$ such that
	\[M\models \forall y<e\bigl(\exists x \theta(x,y)\rightarrow \exists x<f \  \theta(x,y)\bigr),\]
	hence the same is true in $N$ by elementarity. It follows that for some $d$, $N\models \theta(d,b) \wedge d<f$. Any such $d$ must belong to $\sup_N(M)$ which concludes the proof.  It follows that $\sup_N(M)$ is a $\Sigma_n$ elementary initial segment of $N\models \CT_n$. This concludes the proof.
\end{proof}

	Kotlarski and Ratajczyk, in \cite{kotlarski-ratajczyk}, gave a characterisation of arithmetical consequences of $\CT_n$ in terms of the transfinite induction. For each $k\in\mathbb{N}$ define (below, $\alpha^\beta$ denotes ordinal exponentiation):
	\begin{align*}
	\omega_0(k) &= k\\
	\omega_{m+1}(k) &= \omega^{\omega_m(k)}.
	\end{align*}
	Assume a standard coding of ordinals below $\phi_2(0)$ (see e.g. \cite{franzen}) and denote by $\TI(\alpha, \phi)$ the transfinite induction for $\phi$ up to $\alpha$, the formula
	\[\forall \beta\bigl(\forall \gamma\prec \beta \phi(\gamma) \rightarrow \phi(\beta)\bigr)\longrightarrow \forall \gamma\prec \alpha \phi(\gamma).\]
	
	\begin{theorem}[Kotlarski, Ratajczyk, \cite{kotlarski-ratajczyk}]
		For every $n$, the sets of arithmetical consequences of $\CT_n$ and $\PA + \set{\TI(\varepsilon_{\omega_n(k)}, \phi)}{\phi\in\Lang_{\PA}, k\in\mathbb{N}}$ coincide.
	\end{theorem}
	
	\begin{cor}
		For every $n$, the sets of arithmetical consequences of $B\Sigma_{n+1}(T)$ and $\PA + \set{\TI(\varepsilon_{\omega_n(k)}, \phi)}{\phi\in\Lang_{\PA}, k\in\mathbb{N}}$ coincide.
	\end{cor}
	
\section{End-extensions of models of $\CT^-$}

	One obvious strategy to show that $\CT^- + \Coll$ is conservative over $\PA$ would be to show that any countable model of $\CT^-$ has a countable end-extension
	 and thus build an $\omega_1$-chain of models of $\CT^-$. However, in general this strategy is doomed to fail, as witnessed by the following result of Smith (\cite{smith}, Theorem 4.3).
	
	\begin{theorem}[Smith] \label{tw_smith_o_rozszerzeniach_koncowych}
		There exists a countable model of $\CT^-$ which has no end-extension.
	\end{theorem}

	\begin{proof}[Sketch of a proof]
		Take a model $(M,T) \models \CT^-$ in which there is a formula $\phi$ such that $\set{\tuple{x,y} \in M^2}{(M,T) \models T\phi(\num{x},\num{y})}$ is a bijection from $M$ to its proper initial segment $[0,a]$. (It can be shown that such a predicate $T$ exists by a modification of Enayat--Visser argument. A more complete argument can be found in the paper \cite{smith} of Smith.)
		
		Now, all the following sentences are in $T$:
		\begin{enumerate}
			\item $\forall x,y \Big(\phi(x,y) \rightarrow y< \num{a}\Big).$
			\item $\forall x_1, x_2,y \Big( \phi(x_1,y) \wedge \phi(x_2,y) \rightarrow x_1 = x_2\Big)$.
			\item $\forall x, y_1, y_2 \Big( \phi(x,y_1) \wedge \phi(x,y_2) \rightarrow y_1 = y_2 \Big)$.
			\item $\forall x \exists y \phi(x,y).$
		\end{enumerate}
		Therefore, if $(M,T) $ has an end-extension $(N,T')$, then all the above sentences will be in $T'$. This means that in $(N,T')$, the formula $\phi$ defines a bijection from $N$ to $[0,a]$. However, this is impossible, since all elements in $[0,a]$ are already values of elements from $M$ under this bijection.
	\end{proof}
	
	In the light of the above result, once could hope to find some extension $\Theta$ of $\CT^-$ such that:
	\begin{enumerate}
		\item $\Theta$ is conservative over $\PA$.
		\item Any countable model from $\Theta$ has an end-extension to a model of $\Theta$.
		\item $\Theta$ is closed under taking unions of end-extensions. 
	\end{enumerate}
	Our initial hope was that $\CT^- + \INT$ may fit the bill. Note that the lack of internal induction (or, in fact, of internal collection) is exactly the obstruction which makes it impossible for a model introduced by Smith to have an end-extension. Unfortunately, we did not manage to settle the question whether any countable model $(M,T) \models \CT^- + \INT$ has an end-extension to a model of $\CT^- + \INT.$ However, we managed to obtain the following partial result:
	
	\begin{theorem} \label{tw_modele_ctminus_int_rek_nas}
		If $(M,T) \models \CT^- + \INT + \StrReg$ is a countable model recursively saturated in the extended language, then there exists an end extension $(N,T') \models \CT^- + \INT + \StrReg$ which is also recursively saturated in the extended language.
	\end{theorem}
	Before we proceed to the proof, we need one more lemma:
	\begin{lemma} \label{lem_rozszerzenia_model_ctminus_int}
		Let $(M,T,I) \models \CT^- \res I + \StrReg + \INT$ be any countable model recursively saturated in the expanded language. Then there exists $T' \supset T$ such that $(M,T) \models \CT^- + \StrReg+ \INT$.
	\end{lemma}
	The lemma is proved by using resplendence and a slight modification of the Enayat--Visser argument. We will briefly discuss its proof in the Appendix.
	\begin{proof}[Proof of Theorem \ref{tw_modele_ctminus_int_rek_nas}.]
		Let $(M,T) \models \CT^- + \INT + \StrReg$ be countable and recursively saturated. As in the proof of Lemma \ref{lem_ct_minus_ograniczone_z_kolekcja_jest_konserwatywne}, let $X_{T\phi}, \phi \in \form_{\LPA}(M)$ be the disintegration of $T$. The model $(M,X_{T\phi})_{\phi \in M}$ satisfies full induction scheme, so it has an elementary end extension to a model $(N,X'_{T\phi})_{\phi \in M}.$ As in the proof of Lemma \ref{lem_iteracja_lematu_o_kolekcji}, from the predicates $X'_{\phi}$, we can obtain a predicate $T'$ such that $(N,T')\models\CT^- \res M + \INT+ \StrReg$. (One can check with a direct elementarity argument that the construction preserves $\INT$.) In particular, $N$ is recursively saturated. Observe that even if $(N, X'_{T\phi})_{\phi\in M}$ is recursively saturated in the expanded language (it can taken to be so), $(N, T')$ need not be, hence Lemma \ref{lem_rozszerzenia_model_ctminus_int} need not apply. However, by a resplendence argument, we can show that we can find $(N',T'')$ so that:
		\begin{itemize}
			\item The model $(N',M,T,T'')$ is recursively saturated.
			\item $M \prec_e N'$.
			\item $T \subset T''$.
			\item $(N',T'',M) \models \CT^- \res M + \INT + \StrReg$.
			\item $(M,T) \models \CT$. 
		\end{itemize}
More precisely, by Paris--Friedman Theorem \ref{th_paris_friedman}, $M$ and $N$ are isomorphic. Therefore, in $N$ there is a predicate $T^*$ such that $(N,T^*) \simeq (M,T)$. Let $f$ be an isomorphism between these models. The structure $(N,T^*,M,T,f,T')$ witnesses that the following theory is consistent with $\ElDiag(N,T^*)$ using additional predicates $I,T_I,g,\widetilde{T}$:
\begin{itemize}
	\item  $I$ is an elementary initial segment of $N$.
	\item $g:(N,T^*) \to (I,T_I)$ is an isomorphism.
	\item $(N,\widetilde{T}) \models \CT^- \res I + \INT + \StrReg.$
	\item $\widetilde{T} \supset T_I$.
\end{itemize}
In order to see that the theory is consistent, identify $I$ with $M$, $T_I$ with $T$, $g$ with $f$, and $\widetilde{T}$ with $T'$. By resplendence, it can be realised by interpreting $I,T_I,g,\widetilde{T}$ as relations in $N$ in such a way that the obtained model is recursively saturated. Since $(I,T_I)$ is isomorphic with $(M,T)$, the latter model has an end extension $(N',T'')$ with the desired properties.

	\begin{comment}
	However, by Paris--Friedman Theorem \ref{th_paris_friedman},
	the models $M$ and $N$ are isomorphic. So the model $(N,M,T,T')$ witnesses that $\ElDiag(M,T)$ is consistent with the following theory using additional symbols $f,I,T^*$: 
	\begin{itemize}
	\item $I$ is an elementary initial segment of $M$.
	\item $f:M \to I$ is an isomorphism.
	\item $(M,T^*) \models \CT^- \res I + \INT + \StrReg$.
	\item $T^* \supset f[T]$.
	\end{itemize}
	To see that this theory is consistent, we identify $(M,I,f[T],T^*)$ with $(N,M,T,T')$. 
	
	By chronic resplendence, this theory can be interpreted in $M$ in such a way that $(M,I,f[T],T^*)$ is recursively saturated. Again using isomorphism between $(N,M)$ and $(M,I)$, we see that there exists $T'' \supset T$ such that $(N,T'') \models \CT^- \res M + \INT + \StrReg$ and $(N,M,T,T'')$ is recursively saturated.
	
	By Lemma \ref{lem_rozszerzenia_model_ctminus_int}, we can find $T''' \supset T''$ such that $(N,T''') \models \CT^- + \INT + \StrReg$ and we can additionally require that it is recursively saturated. This concludes the proof. 
	\end{comment}
	
	\end{proof}
	
	We could hope that we could build $\omega_1$-like models of $\CT^-$ by taking chains of recursively saturated models of $\CT^- + \INT$. Unfortunately, there is a serious obstruction to this strategy: a union of recursively saturated models need not be recursively saturated and at this point we do not see a clear strategy to circumvent this problem.
	
	Another possible strategy which one could consider is to show that if $M \preceq_e N$ and $T \subset T'$ such that $(M,T) \models \CT^- + \INT $, $(N,T') \models \CT^- \res M + \INT$, then $T'$ can be extended to a predicate $T'' \supset T'$ such that $(N,T'') \models \CT^- + \INT$. I.e., one could hope that we can get rid of the resplendence argument in the above proof. However, we unfortunately know that this is in general impossible without further assumptions. Indeed, there exist countable models $M,N$ and predicates $T,T'$ such that:
	\begin{itemize}
		\item $M \preceq_e N$.
		\item $(M,T) \models \CT^- + \StrReg + \INT$.
		\item $(N,T') \models \CT^- \res M + \StrReg +\INT$.
		\item $T \subset T'$, 
	\end{itemize}
	in which $T'$  cannot be further extended to a predicate $T''$ satisfying $\CT^-$. The proof of this fact will appear in \cite{WcisloKossak}.
	
	Let us make one last remark: the example given by Smith shows how a model can fail to have an end-extension because of how its truth predicate looks like and it has nothing to do with the structure of the underlying arithmetical model. However, quite surprisingly if $(M,T)$ is a model of $\CT^-$, possibly uncountable, then $M$ has an elementary end-extension to a model $N$ which then can be expanded to a model of $\CT^-$. This has been observed by Albert Visser.\footnote{The original argument was slightly different, since it did not use the results of \cite{enayatlelykwcislo}. We are grateful for his permission to include here the proof of this unpublished result.} In the proof, we use the following result, originally proved essentially in \cite{lelyk_wcislo_models}. A (hopefully more perspicuous)  proof of this exact statement will appear in \cite{WcisloKossak}.\footnote{Actually, in order to prove the result below, we only need to have a predicate $T'$ which satisfies uniform Tarski biconditionals for standard formulae and full induction scheme for the extended language. This is exactly what Theorem 4.1 in \cite{lelyk_wcislo_models} gives us. However, we wanted to use a formulation more in line with the notation of this paper.}

	\begin{theorem} \label{tw_w_modelach_ctminus_jest_tez_utb}
		Let $(M,T) \models \CT^-$. Then there exists a $T'$ and a nonstandard $c \in M$ such that $(M,T') \models \CT \res [0,c]$.
	\end{theorem}
	\begin{theorem}[Visser]
		Suppose that $(M,T) \models \CT^-$. Then there exists an elementary end-extension $M \preceq_e N$ and a $T'$ such that $(N,T') \models \CT^-$. 
	\end{theorem}
	\begin{proof}[Sketch of a proof]
		Let $(M,T_0) \models \CT^-$ and let $(M,T) \models \CT \res [0,c]$ with $c$ nonstandard which exists by Theorem \ref{tw_w_modelach_ctminus_jest_tez_utb}. 
		
		Consider the $\Lang_{\PA}\cup\{T\}$-definable set $\Theta$ containing the compositional axioms of $\CT^-$ and all arithmetical sentences $\phi$, possibly nonstandard, such that $\phi \in T$. Within $\PA$, we can formalise the Enayat--Visser conservativity proof for $\CT^-$ over $\PA$ and show that $\Theta$ is consistent. More precisely, for every $n$,  (the straightforward arithmetisation of) the following assertion is provable in $\PA$ (see \cite{enayatlelykwcislo}, Lemma 4.3 which is formulated for the language of arithmetic but generalises to other theories with full induction extending $\PA$):\footnote{Different approaches to the formalisation of the conservativity of $\CT^-$ within $\PA$ were presented in \cite{leigh} and \cite{enayatvisser2}.}
		%%%%%%%%%%%%%%% BW 13.02 Nie pomyliłem się powyżej? Miałeś na myśli ten lemat, w numeracji JSL?
		\begin{center}$(*)$
		If $K\models I\Delta_0+\exp$ is any $\Delta_n$ model, then there exists a $\Delta_{n+1}$ model $(N,T')\models \Th$ and $K$ is an $\Lang_{\PA}$-elementary submodel of $N$. 
		\end{center} 
		where $\Th$ denotes the compositional axioms for the truth predicate from $\CT^-$. Using cut-elimination one checks that (the set of sentences satisfying) $T$ is consistent and then applies $(*)$ to a $\Delta_2$ definable model $K$ of $T$ (viewed as a set of  sentences). We can fix $(N,T')$ given by the above claim.

		 Such $(N,T')$ satisfies $\CT^-$, $N$ is an end-extension of $M$ and for every standard formula $\phi(x_1, \ldots, x_n)$ and elements $a_1, \ldots, a_n \in M$, $N \models \phi(a_1, \ldots, a_n)$ iff $(M,T) \models T \phi(\num{a_1}, \ldots, \num{a_n})$ iff $M \models \phi(a_1, \ldots, a_n)$. This guarantees that $M \preceq N$ which concludes the proof.\footnote{More information on Arithmetised Completeness and the fact that inner models in models of $\PA$ give rise to end-extensions may be found in \cite{kaye}, Section 13.2.}
	\end{proof}
	
	Note that in the above theorem, we do not make any assumptions on the cardinality of $M$. For countable models, the theorem may be proved in a much easier way, since by Lachlan's Theorem \cite{lachlan}, for every model $(M,T)\models \CT^-$, the underlying arithmetical model $M$ is recursively saturated. By a theorem of Friedman, every countable recursively saturated model $M$ has an elementary end extension to another such model $N$.\footnote{This can be seen as follows: by resplendence, $M$ has an initial segment $I \preceq M$ such that $I$ is also recursively saturated. By Paris--Friedman Theorem \ref{th_paris_friedman}, $I \simeq M$. Since $M$ is isomorphic to $I$ and $I$ has an elementary recursively saturated end-extension, the same is true for $M$.} This model, in turn, can be expanded to a model of $\CT^-$ by resplendence of countable recursively saturated models, since $\CT^-$ is conservative over $\PA$ by Theorem \ref{th_kkl}.

	\section{Appendix}
	
	In this article, we made use of some facts which relied on modification of the Enayat--Visser proof. The required changes are rather straightforward, and therefore we moved the proofs to the Appendix. We tried to make the presentation reasonably self-contained, but the reader might want to consult the original paper \cite{enayatvisser2}.
	
	Let us begin with a proof of Lemma \ref{lem_enayat_visser_z_kolekcja}. We restate it for the convenience of the reader.
	
	\begin{lemma} 
			Suppose that $(M,T,I) \models \CT^-{\res} I+ \Coll + \StrReg$ is a countable model recursively saturated in the extended language with $I$ an initial segment, possibly empty. Then there exists $T' \supset T$ such that $(M,T') \models \CT^- + \StrReg$ and, moreover, the model $(M,T,X_{T'\phi})_{\phi \in M}$ satisfies full collection scheme, where the family $\{X_{T'\phi}\}_{\phi\in M}$ is the disintegration of $T'$.
	\end{lemma}
	
	\begin{proof}
		Let $(M,T,I) \models \CT^- \res I + \Coll + \StrReg$. We will find an extension:
		\begin{displaymath}
		(M,T,I) \subset (M^*,T^*,I^*,T')
		\end{displaymath}
		such that 
		\begin{enumerate}
			\item $(M,T,I) \preceq (M^*,T^*,I^*)$ is an elementary extension,
			\item $(M^*,T') \models \CT^- + \StrReg$,
			\item $T^* \subseteq T'$ and
			\item for every $\phi_1, \ldots, \phi_n \in \form_{\LPA}(M^*)$, the predicates $X_{T'\phi_i}$ satisfy full collection jointly with $T^*$.
		\end{enumerate}  
		Since $(M,T,I)$ is resplendent, this will conclude our proof. Note that the predicates $X_{T'\phi}$ are not present in the language, but point 4 in the above list can be expressed in terms of the predicate $T'$ alone by quantifying universally over the formulae $\phi_1, \ldots, \phi_n$. (Which is important, since otherwise the resplendence argument would not be valid.)

		We will construct an $\omega$-chain of models $(M_j,T_j,I_j,X^j_{\phi},S_j)_{\phi \in M_{j-1}}$ using auxiliary predicates $X^j_{\phi}$ such that for any $k$, the chain $\{(M_j,T_j, I_j,X^j_{\phi})\}_{j\geq k, \phi \in M_{k-1}}$  is elementary. Finally, we will set $\bigcup M_j = M^*$, $\bigcup I_j = I^*$, $\bigcup T_j = T^*$. The predicates $S_j$ will be satisfaction predicates compositional for formulae from the model $M_{j-1}$ if $j>0$. Finally, we will set:
		\begin{displaymath}
		T' = \set{\phi \in \Sent_{\LPA}(M^*)}{\exists j \ \phi \in \Sent_{\LPA}(M_j) \wedge (\phi, \emptyset) \in S_{j+1}}.
		\end{displaymath}
		
		At the initial step, we set $M_0 = M, I_0 = I, T_0 = T'_0 = T$. By convention $M_{-1} = \emptyset$. At each step, we inductively take $(M_{j+1},T_{j+1},I_{j+1},X^{j+1}_{\phi},S_{j+1})_{\phi \in M_{j}}$ to be the model of the theory $\Theta_{j}$ expanded with extra predicate consisting of the following axioms:
		\begin{itemize}
			\item $\ElDiag(M_j,T_j,I_j,X^j_{\phi})_{\phi \in M_{j-1}}$.
			\item (Compositional axioms) $\Comp(\phi)$, for $\phi \in \form_{\LPA}(M_j)$ which state that $S_{j+1}$ behaves compositionally with respect to $\phi$. Shortly, we will give a more precise definition.
			\item ($T'$ contains $T$) $\forall x T_{j+1}(x) \rightarrow S_{j+1}(x, \emptyset).$
			\item (The sequence $S_j$ stabilises) $S_{j+1}(\phi,\alpha)$, where $\phi \in M_{j-1}$ and $(\phi,\alpha) \in S_j$.
			\item (The definition of $X^{j+1}_\phi$) $X^{j+1}_{\phi}(\alpha) \equiv \alpha \in \Val(\phi) \wedge S_{j+1}(\phi,\alpha)$, $\phi \in \form_{\LPA}(M_{j})$. 
			\item All the instances of the collection scheme in the language with arithmetical symbols and the predicates $T_{j+1}, X^{j+1}_{\phi}$, where $\phi \in \form_{\LPA}(M_j)$.
			\item (Extensionality Axiom) $\forall \phi \in \form_{\LPA} \forall \alpha \in \Val(\phi) \ S_{j+1}(\phi,\alpha) \equiv S_{j+1}(\phi[\alpha], \emptyset)$.
			\item (Structural Regularity Axiom) $\forall \phi, \psi \in \Sent_{\LPA} \Bigl( \phi \approx \psi \rightarrow S_{j+1}(\phi,\emptyset) \equiv S_{j+1}(\psi,\emptyset)\Bigr).$ 
		\end{itemize}
	Recall that the structural equivalence relation $\phi \approx \psi$ was introduced in Definition \ref{def_structural_equivalence}. Notice one important (but admittedly subtle) difference between this formulation and the original proof of \cite{enayatvisser2}. In the original version, Enayat and Visser required that the constructed chain be elementary only in the signature of the base language. Here, we additionally require elementarity with respect to the predicates $X^j_{\phi}$. . Let us explain a bit what actually happens.
	
	In the $j$-th step, we introduce a predicate $S_{j+1}$ which is compositional for the formulae from the current model $M_j$. Together with this model, we introduce a family of predicates $X^{j+1}_{\phi}$ which are the disintegration of the predicate $S_{j+1}$, but defined only for formulae in the current model $M_j$. Notice that by elementarity, we actually require that $T_{j+1}$ behaves like $T_j$, $I_{j+1}$ behaves like $I_j$, and crucially, $X^{j+1}_{\phi}$ behaves like $X^j_{\phi}$ whenever $\phi \in M_{j-1}$. In other words: we really do require more regularity that in the usual Enayat--Visser construction. For instance, if $\phi(v) \in \form_{\LPA}(M_j)$ has only one free variable, and $x$ is the smallest element satisfying $\phi$ under $S_{j+1}$, (the smallest element such that $S_{j+1}(\phi,\alpha)$ holds, where $\alpha(v) = x$), then by elementarity requirement for $X^{j+1}_{\phi}$, $x$ will stay the smallest such element throughout the whole construction. On the other hand, in the original construction of Enayat--Visser, we essentially only require that $S^{j+1}(\phi,\alpha)$ \emph{still holds} in the later stages. This is enough to guarantee that the compositional conditions hold in the final model, but would not suffice to guarantee that the disintegration of the final model satisfies full collection. To this end, we need some elementarity, and we introduce the predicates $X^j_{\phi}$ to conveniently describe what amount of elementarity is needed. One last remark for the scrupulous reader: from the strict reading of our notation, it follows that there are lots of predicates $X^j_c$, where $c$ is not a formula. We keep them, as we do not want to overload our notation, but they are harmless to the construction.

	The compositional axioms $\Comp(\phi)$ are defined as the conjunction of the following formulae:
		\begin{itemize}
			\item $\forall s,t \in \Term_{\LPA} \forall \alpha \in \Val(\phi) \Big(\phi = (s=t) \rightarrow \big( S_{j+1}(s=t,\alpha) \equiv s^{\alpha} = t^{\alpha} \big) \Big).$
			\item $\forall \psi \in \form_{\LPA} \forall \alpha \in \Val(\phi) \Big( \phi = \neg \psi \rightarrow \big( S_{j+1}(\phi,\alpha) \equiv \neg S_{j+1}(\psi,\alpha) \big)\Big)$.
			\item $\forall \psi, \eta \in \form_{\LPA} \forall \alpha \in \Val \phi \Big(\phi = \psi \vee \eta \rightarrow \big(S_{j+1}(\phi, \alpha) \equiv S_{j+1}(\psi, \alpha) \vee S_{j+1}(\eta,\alpha) \big)\Big)$.
			\item $\forall \psi \in \form_{\LPA}\forall v \in \Var \forall \alpha \in \Val(\phi) \Big( \phi = (\exists v \psi) \rightarrow \big( S_{j+1}(\phi,\alpha) \equiv \exists \beta \sim_v \alpha S_{j+1}(\psi,\beta) \big)\Big)$. 
		\end{itemize}	   
	Note that we officially work in a language without conjunction or universal quantifiers. This choice is simply for convenience and does not affect our results. 
	
	By direct and simple verification, we check that if our construction can be performed, the resulting model $(M^*,T^*,I^*,T')$ satisfies our requirements, i.e.:
	\begin{itemize}
		\item $(M^*,T^*,I^*) \models \ElDiag(M,T,I)$.
		\item $(M,T') \models \CT^-$. 
		\item $T^* \subseteq T'$.
		\item Full collection scheme holds for the arithmetical language expanded with the predicates $T^*$, $X_{T'\phi}$, where $\phi \in \form_{\LPA}(M^*)$ (where $T'\phi$ are the disintegration of $T'$).
	\end{itemize}
In order to verify the last item, notice that every such collection axiom contains only finitely many predicates $X_{T'\phi}$. A model with finitely many such predicates is a union of the elementary chain of models containing the predicates $X^j_{\phi}$. In our construction, we guaranteed that collection scheme holds for the arithmetical language expanded with finitely many such predicates and the predicates $T_j$ corresponding to $T^*$. Therefore, full collection scheme holds by elementarity. 

So the only thing which we need to check is whether our construction can indeed be performed. In other words, we need to verify whether the theories $\Theta_{j}$ can be inductively shown to be consistent. This will be proved in a separate lemma.
	\end{proof}

\begin{lemma} \label{lem_niesprzecznosc_teoryj_EV}
	All theories $\Theta_j$ defined in the above proof of Lemma \ref{lem_enayat_visser_z_kolekcja} are consistent. 
\end{lemma}
In order to deal with the regularity axioms in the following proof, it will be handy to have some extra notation. Let $\phi, \psi \in \form_{\LPA}, \alpha \in \Val(\phi), \beta \in \Val(\psi)$. We say that $(\phi,\alpha), (\psi,\beta)$ are \df{structurally equivalent} iff $\phi[\alpha] \approx \psi[\beta]$, i.e., the sentences $\phi[\alpha]$ and $\psi[\beta]$ are structurally equivalent. We will also denote this relation $(\phi,\alpha) \approx (\psi,\beta)$. Recall that we introduced this notion in Definition \ref{def_structural regularity}. Recall that we call $\phi$ and $\psi$ structurally similar iff they have the same template $\widehat{\phi}$ (see Definition \ref{def_structural_template}.) If $\phi$ and $\psi$ differ only by renaming bounded variables without making any free variable bounded, we say that $\phi$ and $\psi$ are \df{$\alpha$-similar} and denote it with $\phi \simeq \psi$. We are extremely sorry for the amount of notation we need to introduce which has deceptively similar meaning. After these preliminary remarks, we can proceed to the proof.
\begin{proof} [Proof of Lemma \ref{lem_niesprzecznosc_teoryj_EV}]
	In the proof we assume that $j>0$. The case $j=0$ is handled in a similar fashion with a slightly simpler argument. Fix a model $(M_j,T_j,I_j,S_j, X^j_{\phi})_{\phi \in M_{j-1}}\models \Theta_{j}$. We will argue by compactness that the theory $\Theta_{j+1}$ defined using this model is consistent. (The model determines the theory $\Theta_{j+1}$ via its elementary diagram.)  
	
	Fix any finite $\Theta \subset \Theta_{j+1}$. There are only finitely many $\phi \in \form_{\LPA}(M_j)$ either occurring under the predicate $S_{j+1}$ or as an index of the predicate $X^{j+1}_{\phi}$. Let us enumerate these formulae as $\phi_1, \ldots, \phi_n$. It is enough to find an interpretation of the predicate $S_{j+1}$ in the model $(M_j,T_j,I_j)$ such that: 
	\begin{itemize}
		\item $T_j \subset S_{j+1}$.
		\item $S_{j+1}$ respects regularity axioms.
		\item $S_{j+1}$ respects compositional conditions on formulae $\phi_1, \ldots, \phi_n$.
		\item $S_{j+1}$ together with $T_j$ satisfies full collection scheme.
		\item $S_{j+1}(\phi,\alpha) \equiv S_j(\phi,\alpha)$ holds whenever $\phi \in  M_{j-1}$ and $\phi$ is among $\phi_1, \ldots, \phi_n$.
	\end{itemize}
	Notice that the last item guarantees both that stabilisation condition and elementarity for the language with the predicates $X_{\phi}$ hold.
	
	Consider the equivalence classes $[\phi_i]_{\sim}$ of $\phi_1, \ldots, \phi_n$. Consider the following relation $\lhd_0$: $[\phi] \lhd_0 [\psi]$ iff there exist $\phi' \in [\phi]$ and $\psi' \in [\psi]$ such that $\phi'$ is a direct subformula of $\psi'$. One can check that $\unlhd$, the transitive closure of $\lhd_0$, is a partial order on classes (since it is a transitive closure of some binary relation, it is enough to check that no loops can occur, which is obvious).
	
	We define the predicate $S_{j+1}$ by induction on $\unlhd.$ If $[\phi]$ is minimal with respect to this ordering, we consider two cases: either $[\phi] \cap M_{j-1} = \emptyset$ or not. If the former holds, we set:
	\begin{displaymath}
	\forall \psi \in [\phi] \forall \alpha \in \Val(\psi) \ \neg S_{j+1}(\psi, \alpha).
	\end{displaymath}
	Thus $\phi$ defines an empty set under the satisfaction predicate. 
	In the latter case, since $[\phi] \cap M_{j-1}$ is nonempty, and the template $\widehat{\phi}$ is definable with a parameter in $M_{j-1}$, it must also be in $M_{j-1}$ by elementarity. Notice that for any $\psi \in [\phi]$, there exists $\bar{s} \in \TermSeq_{\LPA}$ such that $\psi \simeq \widehat{\phi}(\bar{s}).$ Now, for any $\alpha \in \Val(\phi)$, $\psi[\alpha]$ is also an element of $[\phi]$, so there exists the unique $\bar{t}$ such that $\psi[\alpha] \simeq \widehat{\phi}(\bar{t})$ and, consequently, there is (the unique) $\beta \in \Val(\widehat{\phi})$ such that $\psi[\alpha] \approx \widehat{\phi}[\beta]$ (namely, $\beta$ equal to the sequence of values $\bar{\val{t}}$). We set: 
	\begin{displaymath}
	S_{j+1}(\psi,\alpha) \equiv S_{j}(\widehat{\phi},\beta). 
	\end{displaymath}
	Finally, if $\phi \sim \phi_i$ for some $i \leq n$ and $[\phi]$ is not minimal in the ordering $\unlhd$, we inductively define the behaviour of $S_{j+1}$ so that the compositional conditions are satisfied. For instance, if $\phi = \exists v \psi$, we set
	\begin{displaymath}
	S_{j+1} (\phi,\alpha) \equiv \exists \beta \sim_v \alpha S_{j+1}(\psi,\beta).
	\end{displaymath}
	We add to $S_{j+1}$ all pairs $(\phi, \alpha)$ such that $\phi \in \form_{\LPA}(M_j)$, $\alpha \in \Val(\phi)$, and $T_j(\phi[\alpha])$ holds. Having defined $S_{j+1}$, we set $X^{j+1}_{\phi}(\alpha) \equiv X^{j}_{\phi}(\alpha)$ for $\phi \in M_{j-1}$ and define the sets $X^{j+1}_{\phi}$ so as the definition-axiom of $X^{j+1}_{\phi}$ is satisfied for $\phi \in M_j \setminus M_{j-1}$. 
	
	We have to check that our requirements are satisfied. Let us notice that $S_{j+1}$ satisfies full collection scheme, since it is arithmetically definable in the predicates $T_j$ and $X^j_{\phi}$, where $\phi \in M_{j-1} \cap \{\widehat{\phi_1}, \ldots, \widehat{\phi_n}\}$. These predicates satisfy jointly full collection by assumption. This means that $X^{j+1}_{\phi}$ for $\phi \in M_j$ defined using $S_{j+1}$ also satisfy collection.
	
	Since by induction hypothesis $S_j$ satisfied regularity and compositionality axioms, we check that $S_{j+1}$ agrees with $S_j$ on formulae from $M_{j-1}$. Hence the defined structure satisfies $\ElDiag(M_j,T_j,I_j,X^j_{\phi})_{\phi \in \Theta'}$ where $\Theta'$ is the finite set of $\phi$ such that $X_{\phi}$ occurred in the analysed finite theory $\Theta$. 
	
	The obtained structure satisfies compositional axioms, since regularity and compositionality was satisfied on $S_j$ by induction hypothesis and on $T_j$ by elementarity and the assumption that $T$ satisfies $\CT^- \res I + \StrReg$. Compositional axioms are satisfied on formulae which are not structurally similar to the ones in $M_{j-1}$ directly by our construction.
	
	Finally, we check that the regularity axioms are satisfied. We prove this by induction on the order $\unlhd$. If $[\phi]$ is minimal and $[\phi] \cap M_{j-1} \neq \emptyset$, then structural regularity and extensionality conditions follow by the induction hypothesis on $S_j$ and the definition of $S_{j+1}$. These properties follow directly by definition for the minimal $[\phi]$ such that $[\phi] \cap M_{j-1}$ is empty. If $[\phi]$ is not minimal, then we check that they are preserved by extending $S_{j+1}$ compositionally. This is very simple for the negation and disjunction case, so let us check that they preserved in the step for the existential quantifier. 
	
	Suppose that $S_{j+1}$ satisfies structural regularity and extensionality for formulae in $[\psi]$. Let $\phi = \exists v \psi$. 
	
	In order to verify the extensionality condition, we want to check that $S_{j+1}(\phi,\alpha)$ holds iff $S_{j+1}(\phi[\alpha], \emptyset)$ holds. By compositionality and the induction hypothesis, we have the following equivalences:
	\begin{eqnarray*}
		S_{j+1}(\phi,\alpha) & \equiv & \exists \beta \sim_v \alpha \ S_{j+1}(\psi,\beta) \\
		& \equiv & \exists \beta \sim_v \alpha \ S_{j+1}(\psi[\beta],\emptyset)   \\
		& \equiv & \exists x \ S_{j+1}(\psi[\alpha],\{\tuple{v,x}\}) \\
		& \equiv & S_{j+1}(\phi[\alpha],\emptyset).
	\end{eqnarray*}  
	Notice that $\{\tuple{v,x}\}$ is an assignment which sends $v$ to $x$, so  $S_{j+1}(\psi[\alpha],\{\tuple{v,x}\})$ makes sense. Observe that $\psi[\alpha]$ is a formula with only one free variable $v$, the rest of free variables in $\psi$ being filled out by the numerals $\num{\alpha(w)}$.
	
	We verify structural regularity in a similar manner. Let us assume that $\phi \approx \phi'$ and that $\phi = \exists v \psi.$ Then $\phi' = \exists w \psi'$ and there exist sequences $\bar{t}, \bar{s} \in \ClTermSeq_{\LPA}(M_j)$ with the same values such that
	\begin{displaymath}
	\phi \simeq \widehat{\phi}(\bar{t}), \phi' \simeq \widehat{\phi}(\bar{s}). 
	\end{displaymath}  
	 Suppose that $S_{j+1}(\phi,\emptyset)$ holds. By symmetry, it is enough to show that $S_{j+1}(\phi',\emptyset)$ holds as well. 
	 
	 Since $S_{j+1}(\phi,\emptyset)$ holds, by compositionality there exists $\alpha \sim_v \emptyset$ such that $S_{j+1}(\psi,\alpha)$ holds (where $\alpha$ is an assignment with the domain either equal to $\{v\}$ or empty). By extensionality, $S_{j+1}(\psi[\alpha],\emptyset)$ holds. Let $\beta$ be an assignment such that $\beta(w) = \alpha(v)$. It is enough to show that $\psi[\alpha] \approx \psi'[\beta]$, since then $S_{j+1}(\phi',\emptyset)$ follows by structural regularity and compositionality.
	 
	 Let us check that $\psi[\alpha] \approx \psi'[\beta].$ Let $\alpha(v) = \beta(w)=c$. Let $\bar{t}' = \bar{t} \frown \tuple{\num{c}}, \bar{s}' = \bar{s} \frown \tuple{\num{c}}.$ There exists a variable $u$ and a formula $\eta$ such that $\widehat{\phi} = \exists u \eta$ and both $\widehat{\psi}$ and $\widehat{\psi}'$ are equal to $\widehat{\eta}$. We see that 
	 $\psi[\alpha] \simeq \eta(\bar{t}')$ and $\psi'[\beta] \simeq \eta(\bar{s}')$ where  $\bar{s}'$ and $\bar{t}'$ have the same values. On the other hand,  there exist sequences of terms $\bar{s}''$ and $\bar{t}''$ with the same values such that $\eta(\bar{t}') \simeq \widehat{\eta}(\bar{t}'')$ and $\eta(s') \simeq \widehat{\eta}(\bar{s}'')$. Since $\widehat{\eta} = \widehat{\psi} = \widehat{\psi'}$, by definition of structural equivalence this implies $\psi[\alpha] \approx \psi'[\beta]$ thus concluding the proof of Lemma \ref{lem_niesprzecznosc_teoryj_EV}.

\end{proof}

Let us notice that in the above proof, collection was preserved, since the interpretations of $S_{j+1}$ for finitely many new formulae were arithmetically defined in finitely many predicates $X_{\phi}$ and $T_j$. We could run a very similar argument in order to obtain a truth predicate satisfying internal induction $\INT$, assuming that it was satisfied by our initial $T$, thus proving Lemma \ref{lem_rozszerzenia_model_ctminus_int}  and Theorem \ref{th_ctminus_plus_int_plus_strreg_konserwatywne}. Since the argument in that case is essentially the same with no non-trivial modifications required, we omit it. 
	
\section*{Acknowledgements}	
This research was supported by an NCN OPUS grant 2017/27/B/HS1/01830, "Truth theories and their strength."


\begin{thebibliography}{10}
	
	\bibitem{cieslinskict0}
	Cezary Cieśliński.
	\newblock Deflationary truth and pathologies.
	\newblock {\em The Journal of Philosophical Logic}, 39(3):325--337, 2010.
	
	\bibitem{cies}
	Cezary Cieśliński.
	\newblock Truth, conservativeness and provability.
	\newblock {\em Mind}, 119:409--422, 2010.
	
	\bibitem{cies_ksiazka}
	Cezary Cieśliński.
	\newblock {\em The Epistemic Lightness of Truth. {D}eflationism and its Logic}.
	\newblock Cambridge University Press, 2017.
	
	\bibitem{enayatlelykwcislo}
	Ali Enayat, Mateusz Łełyk, and Bartosz Wcisło.
	\newblock Truth and feasible reducibility.
	\newblock {\em to appear in Journal of Symbolic Logic}.
	
	\bibitem{enayatmohsenipour}
	Ali Enayat and Shahram Mohsenipour.
	\newblock Model theory of the regularity and reflection schemes.
	\newblock {\em Archive for Mathematical Logic}, 47:447--464, 2008.
	
	\bibitem{enayatvisser2}
	Ali Enayat and Albert Visser.
	\newblock New constructions of satisfaction classes.
	\newblock In Theodora Achourioti, Henri Galinon, Jos\'e Mart\'inez~Fern\'andez,
	and Kentaro Fujimoto, editors, {\em Unifying the Philosophy of Truth}, pages
	321--325. {S}pringer-{V}erlag, 2015.
	
	\bibitem{lelyk_thesis}
	Mateusz Łełyk.
	\newblock Axiomatic theories of truth, bounded induction and reflection
	principles.
	
	\bibitem{franzen}
	Torkel Franzen.
	\newblock {\em Inexhaustibility{:} an Inexhaustive Treatment}.
	\newblock {A} {K} {P}eters/{CRC} {P}ress, 2004.
	
	\bibitem{hajekpudlak}
	Petr H\'ajek and Pavel Pudl\'ak.
	\newblock {\em Metamathematics of First-Order Arithmetic}.
	\newblock Springer-Verlag, 1993.
	
	\bibitem{halbach}
	Volker Halbach.
	\newblock {\em Axiomatic Theories of Truth}.
	\newblock Cambridge University Press, 2011.
	
	\bibitem{kaye}
	Richard Kaye.
	\newblock {\em Models of {P}eano {A}rithmetic}.
	\newblock Oxford{:} Clarendon Press, 1991.
	
	\bibitem{kaye-slides}
	Richard Kaye and Alexander Jones.
	\newblock Truth and collection in nonstandard models of $\pa$.
	\newblock {\em Midlands Logic Seminar}.
	
	\bibitem{keisler}
	Jerome Keisler.
	\newblock {\em Model Theory for Infinitary Logic}.
	
	\bibitem{WcisloKossak}
	Roman Kossak and Bartosz Wcisło.
	\newblock Disjunctions with stopping condition.
	
	\bibitem{kotlarski}
	Henryk Kotlarski.
	\newblock Bounded induction and satisfaction classes.
	\newblock {\em Zeitschrift f\"ur matematische Logik und Grundlagen der
		Mathematik}, 32:531--544, 1986.
	
	\bibitem{kkl}
	Henryk Kotlarski, Stanisław Krajewski, and Alistair Lachlan.
	\newblock Construction of satisfaction classes for nonstandard models.
	\newblock {\em Canadian Mathematical Bulletin}, 24:283--93, 1981.
	
	\bibitem{kotlarski-ratajczyk}
	Henryk Kotlarski and Zygmunt Ratajczyk.
	\newblock More on induction in the language with a full satisfaction class.
	\newblock {\em Zeitschrift f\"ur mathematische logik}, 36:441--454, 1990.
	
	\bibitem{lachlan}
	Alistair~H. Lachlan.
	\newblock Full satisfaction classes and recursive saturation.
	\newblock {\em Canadian Mathmematical Bulletin}, 24:295--297, 1981.
	
	\bibitem{leigh}
	Graham Leigh.
	\newblock Conservativity for theories of compositional truth via cut
	elimination.
	\newblock {\em The Journal of Symbolic Logic}, 80(3):845--865, 2015.
	
	\bibitem{lelyk_wcislo_models}
	Mateusz {\L}ełyk and Bartosz Wcisło.
	\newblock Models of weak theories of truth.
	\newblock {\em Archive for Mathematical Logic}, 56:453--474, 2017.
	
	\bibitem{wcislyk}
	Mateusz {\L}ełyk and Bartosz Wcisło.
	\newblock Notes on bounded induction for the compositional truth predicate.
	\newblock {\em The Review of Symbolic Logic}, 10:455--480, 2017.
	
	\bibitem{smith}
	Stuart~T. Smith.
	\newblock Nonstandard definability.
	\newblock {\em Annals of Pure and Applied Logic}, 42(1):21--43, 1989.
	
	\bibitem{smorynski}
	Craig Smory\'nski.
	\newblock $\omega$-consistency and reflection.
	\newblock In {\em Colloque International de Logique (Colloq. Int. CNRS),},
	pages 167 -- 181. CNRS Inst. B. Pascal, 1977.
	
\end{thebibliography}
\end{document}